\documentclass{amsart}
\usepackage{xcolor}
\usepackage{mathtools}
\usepackage{amsmath}
\usepackage{cite}
\usepackage{amsthm}
\usepackage{graphicx}
\usepackage{amssymb}
\usepackage{epstopdf}
\usepackage{nicefrac}
\usepackage{mathrsfs}
\usepackage{hyperref}
\usepackage{enumitem}

\usepackage[margin=3cm]{geometry}

\usepackage{tikz}  
\usetikzlibrary{calc,arrows.meta,decorations.pathreplacing}

\theoremstyle{plain}
\newtheorem{theo}{Theorem}[section]
\newtheorem{lm}[theo]{Lemma}

\newtheorem{cor}[theo]{Corollary}

\theoremstyle{plain}

\DeclareMathOperator{\PW}{PW} 
\DeclareMathOperator{\supp}{supp}

\DeclareMathOperator{\BMO}{BMO} 
 
\DeclareMathOperator{\produ}{prod} 
\newcommand{\R}{\mathbb R}
\newcommand{\F}{\mathcal F}
\newcommand{\Finve}{\mathcal F^{-1}}
\newcommand{\wtOmega}{\widetilde\Omega}

\title{Recovering Product BMO from Schatten Hankel operators}

\author[K. Bampouras]{Konstantinos Bampouras}
\address{Sabanci University Tuzla Campus, Orta Mahalle, Üniversite Cadesi No:27 Tuzla, 34956 Istanbul, Turkey}
\email{kostasbaburas@gmail.com}

\author[K.-M. Perfekt]{Karl-Mikael Perfekt}
\address{Department of Mathematical Sciences, Norwegian University of Science and Technology (NTNU), 7491 Trondheim, Norway}
\email{karl-mikael.perfekt@ntnu.no}


\date{\today}
\begin{document}
	\begin{abstract}
		We prove that if a small Hankel operator on the product Hardy space belongs to some Schatten class $S^p$, $p < \infty$, then it has a symbol in product $\BMO$. In other words, the conclusion of Nehari's theorem holds under the hypothesis that the operator belongs to a Schatten class.
	\end{abstract}
	
	\maketitle
	
	\section{Introduction}
	
	The classical theory of singular integral operators is notoriously difficult to extend to multi-parameter operators \cite{Fef86, Pipher86, Rud69}. Many basic results require significantly more involved proofs \cite{AMPS23, Journe85, MPV22}, and other pillars of the traditional theory remain as open problems already in the bi-parameter setting. One of the main outstanding difficulties can be summarized in terms of the class of small Hankel operators $H_\varphi \colon H^2(\mathbb{C}_+^2) \to H^2(\mathbb{C}_+^2)$ on the product Hardy space: if $H_\varphi$ is bounded, must it be that the symbol $\varphi$ belongs to Chang--Fefferman product $\BMO_{\produ}$? In other words, is there a lower bound
	\begin{equation} \label{eq:bmolowerbound}
		\|\varphi\|_{\BMO_{\produ}} \lesssim \|H_\varphi\|?
	\end{equation}
	
	Some definitions are in order. We will take a boundary-value view of the analytic Hardy spaces,
	$$H^p(\mathbb{C}_+^2) = \{f \in L^p(\R^2) \, : \, \supp \F f \subset \R_+^2\}, \qquad 1 \leq p < \infty,$$
	where $\mathcal{F}$ is the distributional Fourier transform and $\R_+ = [0,\infty)$, but still use the notation $H^p(\mathbb{C}_+^2)$ in order to emphasize that we are dealing with the product Hardy space. We express that $\supp \F f \subset \R_+^2$ by saying that $f$ is analytic, and let $P_{++} \colon L^2(\R^2) \to H^2(\mathbb{C}_+^2)$ be the analytic projection 
	$$P_{++} = \Finve \chi_{\R_+^2} \F.$$ 
	Given a symbol $\varphi$, a Hankel operator $H_\varphi \colon H^2(\mathbb{C}_+^2) \to H^2(\mathbb{C}_+^2)$ is of the form
	$$\F (H_\varphi f)(\xi) = \int_{\mathbb{R}_+^2} \hat{\varphi}(\xi+ \lambda) \hat{f}(\lambda) \, d\lambda, \qquad \xi \in \mathbb{R}_+^2, $$
	where $\hat{f} = \F f$. In other words, $H_\varphi = P_{++} ( \varphi Jf)$, where $J$ is the linear involution $Jf(x) = f(-x)$. Functions $g \in \BMO_{\produ}$ are those that have a representation $g = b_0 + H_1b_1 + H_2b_2 + H_1H_2 b_3$, where $b_j \in L^\infty$, $0 \leq j \leq 3$ and $H_1, H_2$ are the partial Hilbert transforms in $x_1$ and $x_2$. However, for our purpose, the important fact will be that the dual space 
	\begin{equation} \label{eq:H1duality}
		(H^1(\mathbb{C}_+^2))^\ast = P_{++} \BMO_{\produ}.
	\end{equation}
	Note that $P_{++} \colon BMO_{\produ} \to BMO_{\produ}$ is bounded, by the rudiments of product Calderón--Zygmund theory \cite{FefStein82, Journe85}, and so it will not be a restriction to assume that the symbol $\varphi$ is analytic to begin with. By the duality \eqref{eq:H1duality}, it is then clear (and well known) that \eqref{eq:bmolowerbound} is equivalent to the analogue of Nehari's theorem in this setting: if $H_\varphi$ is bounded, there is $b \in L^\infty(\mathbb{R}^2)$ such that $P_{++} b = \varphi$. 
    
    Despite being one of the most central problems of the multi-parameter theory, verifying the validity of \eqref{eq:bmolowerbound}, and thus of Nehari's theorem, still stands today as one of the major challenges of the subject.  We refer to Holmes, Treil, and Volberg \cite{HTV21} for an overview of the state-of-the-art. In \cite{HTV21} they also obtained a positive result in a restricted setting, by imposing a sparsity assumption on the Haar--Fourier coefficients of the symbol.
	
	In this paper, we explore the problem in a different direction. We will prove that the conclusion of Nehari's theorem holds if one strengthens the hypothesis on the operator. To appreciate this, note, by a routine approximation argument, that it is sufficient to prove Nehari's theorem for the class of compact Hankel operators $H_\varphi \in S^\infty$. Hence the full Nehari theorem appears formally as the limiting case $p = \infty$ of the following theorem.
	\begin{theo}
		Suppose that $\varphi$ is an analytic symbol such that
		$$H_\varphi \in \bigcup_{p \geq 1} S^p(H^2(\mathbb{C}_+^2)).$$
		Then $\varphi \in \BMO_{\produ}$. In other words, there exists $\psi \in L^\infty(\mathbb{R}^2)$ such that $P_{++} \psi = \varphi$. 
		
		Furthermore, the lower bound 
		$$\|\varphi\|_{\BMO_{\produ}} \lesssim_{p} \|H_\varphi\|_{S^p}$$
		holds, with a constant possibly depending on $p$.
	\end{theo}
	
	Let us make a few remarks on this theorem. It is straightforward to verify that we could have alternatively stated the theorem for iterated commutators $[H_1, [H_2, \varphi]] \in S^p(L^2(\R^2))$. As far as we know, the best available previous result in this direction was that $H_\varphi \in S^2(H^2(\mathbb{C}_+^2))$ implies the existence of a bounded symbol; after a translation, this statement with optimal constant is precisely the content of Helson's inequality \cite[Eq. (2.10)]{MR2263964}. The proof of our theorem is by necessity entirely factorization-free, which may be of some independent interest even in one dimension: it is well known that Nehari's theorem is equivalent to the weak factorization $H^1(\mathbb{C}_+^2) = H^2(\mathbb{C}_+^2) \odot H^2(\mathbb{C}_+^2)$.  We also expect our theorem to be valid for Hankel operators $H_\varphi \colon H^2(\mathbb{C}_+^n) \to H^2(\mathbb{C}_+^n)$ when $n \geq 3$. However, this requires further proof, as the geometry of open sets $\Omega \subset \R^n$ is substantially more involved, cf. \cite{Pipher86}. We have therefore chosen to focus on $n=2$ in this paper.
	
	Let us describe some of the ingredients in the proof of our theorem. First of all, it uses a Besov space description of the operators $H_\varphi \in S^p$. The relevant Besov spaces $B_{p,p}^{1/p}(\mathbb{C}_+^2)$ do not sit in the usual family of Besov spaces for $\R^2$, but are instead obtained by tensorization, reflecting the product structure of our spaces. We will also rely on atomic decompositions of the product Hardy space. Such a decomposition was first given by Chang and Fefferman, with $L^2$-atoms adapted to a general open set $\Omega$, but we require $L^q$-normalized atoms, for every $1 < q < 2$. In this latter case, the normalization condition necessarily involves measurements of the embeddedness levels of the dyadic rectangles contained in $\Omega$. The familiar reader recognizes that this reflects the famous geometric packing estimate known as Journ\'e's lemma. We refer to Subsections~\ref{subsec:journe} and \ref{subsec:atomic} for further discussion and references. A final ingredient to mention is a local summability estimate over dyadic rectangles contained in $\Omega$, penalized by their embeddedness constants, see \eqref{eq:maingeometric}.
	
	As a final point of the introduction, let us mention that it is also natural to study (truncated) Hankel operators $H_\varphi \colon \PW(\Omega) \to \PW(\Omega)$ on the Paley--Wiener space associated with a bounded convex domain $\Omega \subset \R^n$. In general, the situation is then much less clear. For instance, in \cite{bampouras2026nehari}, it has been shown that for every $p>4$ there exists a Schatten class Hankel operator $H_\varphi\in S^p(\PW(\mathbb{D}))$ that does not admit a bounded symbol.
	However, by following the proof of \cite[Proposition 5.2]{MR4227573} more or less verbatim, we obtain the following corollary of our theorem.
	\begin{cor}
		Let $P\subset \mathbb{R}^2$ be a polygon and let $H_\varphi$ be a Hankel operator on $\PW(P)$ with $$H_\varphi \in \bigcup_{p\geq 1}S^p(\PW(P)).$$ Then there exists a bounded function $\psi \in L^\infty(\mathbb{R}^2)$ such that $H_\psi =H_\varphi.$
	\end{cor}
	
	\textbf{Acknowledgments.} The authors used ChatGPT 5.4 by OpenAI as an auxiliary tool for literature review, exploration of potential approaches to the problem, proofreading, and identifying minor mathematical mistakes in earlier drafts. The authors independently wrote and verified all parts of the paper. The second author was supported by grant no. 334466 of the Research Council of Norway, “Fourier
	Methods and Multiplicative Analysis”.
	\section{Preliminaries}
	\subsection{The Besov space $B_{p,p}^{1/p}(\mathbb{C}_+^2)$}
	To define the Besov space $B_{p,p}^{1/p}(\mathbb{C}_+^2)$, fix a function $\psi \in C_c^\infty(\R_+)$ such that
	\[
	\supp \psi \subset [1/2,2], \qquad \sum_{j\in\mathbb Z}\psi(2^{-j}\xi)=1, \; \xi > 0.
	\]
	Furthermore, let $\widetilde\psi \in C_c^\infty(\R_+)$ be another function with $\supp \widetilde \psi \subset [1/4, 4]$ such that $\widetilde\psi(\xi)=1$ for all $\xi\in \supp \psi$.
	We let
	\[
	\psi_j(\xi):=\psi(2^{-j}\xi),\qquad
	\widetilde\psi_j(\xi):=\widetilde\psi(2^{-j}\xi),
	\]
	and define accordingly the one-parameter Fourier multipliers
	\[
	\F(\Delta_j^{(1)} f)(\xi_1,\xi_2)=\psi_j(\xi_1)\widehat f(\xi_1,\xi_2),
	\qquad
	\F(\widetilde\Delta_j^{(1)} f)(\xi_1,\xi_2)=\widetilde\psi_j(\xi_1)\widehat f(\xi_1,\xi_2),
	\]
	and, similarly in the second variable,
	\[
	\F(\Delta_j^{(2)} f)(\xi_1,\xi_2)=\psi_j(\xi_2)\widehat f(\xi_1,\xi_2),
	\qquad
	\F(\widetilde\Delta_j^{(2)} f)(\xi_1,\xi_2)=\widetilde\psi_j(\xi_2)\widehat f(\xi_1,\xi_2).
	\]
	The Besov space $B_{p,p}^{1/p}(\mathbb{C}_+^2)$, $1 \leq p < \infty$, is defined with respect to the convolution operator
	\[
	\Delta_{i,j}:=\Delta_i^{(1)}\Delta_j^{(2)}.
	\]
	It consists of the functions\footnote{A priori, the elements in the Besov space should be considered as distributions, but the theorem shows that they are locally integrable functions.} $f$ with Fourier support in $\R_+^2$ satisfying 
	$$ \|f\|_{B_{p,p}^{1/p}} = \|S_p(f)\|_{L^p(\R^2)} < \infty,$$
	where
	\[
	S_p(f)
	=
	\Bigg(\sum_{i,j\in\mathbb Z}2^{i+j}\,|\Delta_{i,j}f |^p\Bigg)^{1/p}.
	\]
	Note also that for tempered distributions $f$ with Fourier support in $(0,\infty)^2$, we have the reproducing formula
	\begin{equation} \label{eq:Caldron}	
		f=\sum_{i,j\in\mathbb Z}\widetilde\Delta_{i,j}\Delta_{i,j}f,
	\end{equation}
	where $\widetilde\Delta_{i,j}:=\widetilde\Delta_i^{(1)}\widetilde\Delta_j^{(2)}$, since $\sum_{j} \widetilde\psi_j \psi_j=1$ on $\R_+$.
	We will use the following characterization of Schatten class Hankel operators.
	\begin{theo}[\cite{BP25, MR2097606}] \label{thm:besov}
		Let $\varphi$ be an analytic symbol, and $1 \leq p < \infty$. Then $H_\varphi \in S^p(H^2(\mathbb{C}_+^2))$ if and only if $\varphi \in B_{p,p}^{1/p}(\mathbb{C}_+^2)$, with equivalence of norms.
	\end{theo}
	
	\subsection{The Journ\'e lemma} \label{subsec:journe}
	We let $\mathcal{D}$ denote the family of dyadic intervals $J = [m2^n, (m+1)2^n)$, $m, n \in \mathbb{Z}$, and let $\mathcal{D}_2$ denote the family of dyadic rectangles $R = I \times J$, where $I, J \in \mathcal{D}$. For $c > 0$, $cJ$ and $cR$ denote the interval and rectangle concentric with $J$ and $R$, respectively, dilated by the factor $c$. Given an open set $\Omega \subset \R^2$ of finite measure, $M_1(\Omega)$ denotes the collection of dyadic rectangles $R = I \times J \subset \Omega$ for which the interval $I$ is maximal. That is, if $I \times J \subset \hat{I} \times J \subset \Omega$ for another interval $\hat{I} \in \mathcal{D}$, then $I = \hat{I}$. $M_2(\Omega)$ is defined analogously as the collection of dyadic rectangles $R \subset \Omega$ that are maximal in the $x_2$-direction. 
	
	Following Journ\'e, for an open set $\Omega \subset \R^2$ of finite measure, we consider its extension
	$$\wtOmega = \{x \in \R^2 \, : \,  \mathscr{M}_s \chi_\Omega(x) > 1/2\},$$ where $\mathscr{M}_s$ denotes the strong maximal function. By the maximal theorem,
	$$\Omega \subset \wtOmega, \qquad |\wtOmega| \leq C|\Omega|,$$
	where $C > 1$ is an absolute constant. Furthermore, for $R = I \times J \in M_2(\Omega)$, we denote by $\hat I=\hat I(R)$ the maximal dyadic interval such that $\hat I\times J\subset \wtOmega$, and define the embeddedness constant of $R$ by
	\[
	\gamma_1(R)=\frac{|\hat I|}{|I|}.
	\]
	We similarly define the embeddedness constant $\gamma_2(R)$ for $R \in M_1(\Omega)$. 
	
	The Journ\'e lemma \cite{Journe85, Pipher86} gives a weighted packing estimate for general open sets $\Omega$.
	\begin{lm}[Journ\'e]
		If $\Omega \subset \R^2$ is an open set of finite measure, then for every $\delta > 0$,
		$$\sum_{R \in M_2(\Omega)} \gamma_1(R)^{-\delta} |R| \lesssim |\Omega|,$$
		with an implied constant depending only on $\delta$.
	\end{lm} 
	
	\subsection{$L^q$-normalized atomic decompositions} \label{subsec:atomic}
	A famous counterexample given by Carleson in 1974 showed that dual space $(H^1(\mathbb{C}_+^2))^\ast$ cannot be described in terms of the naive extension of the 1-dimensional $\BMO$ space. It immediately implies that it is not possible to give an atomic decomposition of $H^1(\mathbb{C}_+^2)$ based on (projections of) atoms which are adapted only to rectangles. Chang and Fefferman \cite{ChaFef82} then gave us the correct atomic decomposition, where the atoms must be adapted to general open sets $\Omega \subset \R^2$. This precisely explains why the product setting is so intricate and highlights the importance of the Journ\'e lemma.
	
	The atoms of Chang and Fefferman are $L^2$-normalized, but in order to prove our theorem, we will need $L^q$-normalized atoms, where $q < 2$ is the dual index of $p$. A decomposition with $L^q$-normalized atoms, $1 < q < 2$, was given much later by Han, Lu, and Zhao \cite{HanLuZhao10}, and requires an even more involved atomic concept. For $1 < q < 2$, we say that $a$ is an $L^q$-normalized atom adapted to $\Omega$ if 
	$$\supp a \subset \Omega, \qquad \|a\|_{L^q(\R^2)} \leq |\Omega|^{1/q - 1},$$
	and if $a$ has a further decomposition
	$$a = \sum_{R\in M_1(\Omega)} a_R + \sum_{R\in M_2(\Omega)} a_R$$
	into rectangular pieces, $\supp a_R \subset 3R$, satisfying
	\begin{equation} \label{eq:qnormalization}
		\Bigg(
		\sum_{R\in M_1(\Omega)}\gamma_2(R)^{-\delta}\|a_R\|_{L^q}^q
		+
		\sum_{R\in M_2(\Omega)}\gamma_1(R)^{-\delta}\|a_R\|_{L^q}^q
		\Bigg)^{1/q}
		\lesssim_\delta
		|\Omega|^{1/q-1},
	\end{equation}
	for every $\delta > 0$ (the implied constant depends on $\delta$ and $q$ only), and the cancellation conditions
	\begin{equation} \label{eq:cancellation}
		\int_\R a_R(x_1,x_2)\,dx_1=0\quad \textrm{ a.e. }x_2,
		\qquad
		\int_\R a_R(x_1,x_2)\,dx_2=0\quad \text{ a.e. }x_1,
	\end{equation}
	hold for every $R \in M_1(\Omega) \cup M_2(\Omega)$.
	\begin{theo}[\cite{HanLuZhao10, HLPW21}] \label{thm:atomic}
		Let $1 < q < 2$. If $f \in H^1(\mathbb{C}_+^2) \cap L^q(\R^2)$ there exists a sequence of $L^q$-normalized atoms $\{a_i\}$ and scalars $\{\lambda_i\}$ such that 
		$$f = P_{++}\left (\sum_i \lambda_i a_i \right), \qquad c_1 \sum_i |\lambda_i| \leq \|f\|_{H^1} \leq c_2 \sum_i |\lambda_i|,$$
		where the constants $c_1$ and $c_2$ depend only on $q$.
	\end{theo}
	We have added the projection $P_{++}$ in this statement, in order to work directly with the analytic Hardy space $H^1(\mathbb{C}_+^2)$. The original statement is given for the real product Hardy space (which projects to $H^1(\mathbb{C}_+^2)$ under $P_{++}$). 
	\subsection{A basic annular decay estimate} \label{subsec:anndecay}
	Given $I \in \mathcal{D}$, we split $\R$ into annuli in the following way:
	$$E_0(I) = 8 I, \qquad E_u(I) =(2^u I)\setminus (2^{u-1}I), \; u \geq 4.$$
	For convenience of notation, we exclude $u = 1,2,3$ from the definition, but we do not always explicitly spell this out.
	Similarly, for a rectangle $R = I \times J$, we let 
	$$E_{u, v}(R) = E_u(I) \times E_v(J),$$ 
	and suppress $1 \leq u,v \leq 3$ from the notation.
	\begin{lm} \label{lem:annulardecay}
		Let $1 \leq q < \infty$, and let $a_R \in L^q(\R^2)$ be supported in a rectangle $3R$, $R = I \times J$, and satisfy the cancellation conditions \eqref{eq:cancellation}.
		Then for every \(M\ge 1\), 
		\begin{equation*}
			\|\widetilde\Delta_{i,j}^\ast a_R\|_{L^q(E_{u,v}(R))}
			\lesssim
			m_i(I)m_j(J)\,\theta_u(i,I)\,\theta_v(j,J)\,\|a_R\|_{L^q},
		\end{equation*}
		with an implied constant depending only on $M$ and $q$, where 
		$$m_i(I):=\min(1,2^i|I|)$$
		and
		\[
		\theta_0(i,I) = 1,\qquad
		\theta_u(i,I) = (1+2^u2^i|I|)^{-M}, \quad u\ge 4.
		\]
	\end{lm}
	\begin{proof}	
		By iteration, it suffices to show the one dimensional estimate $$	\|\widetilde\Delta_{i}^\ast f\|_{L^q(E_{u}(I))}
		\lesssim
		m_i(I)\theta_u(i,I)\|f\|_{L^q},$$ for functions
		$$f\in L^q(\mathbb{R}), \qquad \supp f\subset 3I, \qquad \int_{\mathbb{R}}f(x) \, dx=0.$$ Note that if $c_I$ is the center of $I$ and we set $g(x)=f(\frac{|I|}{2}x+c_I)$, then $g\in L^q(\mathbb{R})$, $\supp g\subset (-3,3)$, $\int_{\mathbb{R}}g(x)dx=0$ and, if $|I|=2^{i'}$,
		$$\widetilde\Delta_{i}^\ast f(x)=\widetilde\Delta_{i+i'-1}^\ast g\left(2\frac{x-c_I}{|I|}\right).$$
		By a simple computation it therefore suffices to consider $I=(-1,1)$.   We will only explicitly consider the case $u\geq4$; the treatment of $u=0$ is similar. 
		
		Let $K$ be the convolution kernel of $\widetilde\Delta_0^\ast$ and let $K_i(t)=2^iK(2^it)$, $i \in \mathbb{Z}$. For $x\in E_u(-1,1)$ and $y\in (-3,3)$,
		$$|\widetilde\Delta_{i}^\ast f(x)|=\left|\int_{\mathbb{R}}K_i(x-y)f(y)dy \right|\leq \int_{\mathbb{R}}|K_i(x-y)-K_i(x)||f(y)|dy.$$
		Now, using the mean value theorem, we have that
		$$|K_i(x-y)-K_i(x)|\leq |y|\sup_{|\xi-x|<|y|}|K_i'(\xi)|\lesssim 2^{2i}\sup_{|\xi-x|<3}|K'(2^i\xi)|\lesssim 2^{2i}\sup_{\xi\in 2^i(x-3,x+3)}|K'(\xi)|.$$ Since now $K'$ is a Schwartz function, for every $M\geq 2$ it holds that 
		$|K'(t)|\lesssim (1+|t|)^{-M}$, and thus
		$$|K_i(x-y)-K_i(x)| \lesssim  2^{2i}\sup_{\xi\in 2^i(x-3,x+3)}(1+|\xi|)^{-M}\approx 2^{2i}(1+2^i|x|)^{-M},$$ where the last estimate holds since $x\in E_u(-1,1)$ and $u\geq 4$.
		Thus, we conclude that
		\begin{align*} 
			\|\widetilde{\Delta}^\ast_i f\|_{L^q(E_u(-1,1))}&\leq\left(\int_{E_u(-1,1)}\left(\int_{(-3,3)}|K_i(x-y)-K_i(x)||f(y)|dy\right)^qdx\right)^{\frac{1}{q}}
			\nonumber \\
			&\lesssim2^{2i}\left(\int_{E_u(-1,1)}(1+2^i|x|)^{-Mq}\left(\int_{(-3,3)}|f(y)|dy\right)^qdx\right)^{\frac{1}{q}}
			\nonumber \\
			&\lesssim2^{2i}\|f\|_{L^1}\left(\int_{E_u(-1,1)}(1+2^i|x|)^{-Mq}dx\right)^{\frac{1}{q}}
			\nonumber \\
			&\lesssim2^{2i}2^{\frac{u}{q}}(1+2^{u+i})^{-M}\|f\|_{L^q}.
			\nonumber
		\end{align*}
		Now we can observe that $2^{i}2^{\frac{u}{q}}\leq 1+2^{i+u}$ and $2^{2i}2^{\frac{u}{q}}\leq (1+2^{i+u})^2$ and thus
		$$\|\widetilde{\Delta}^\ast_i f\|_{L^q(E_u(-1,1))}\lesssim 2^i(1+2^{u+i})^{-M+1}\|f\|_{L^q}$$
		and
		$$\|\widetilde{\Delta}^\ast_i f\|_{L^q(E_u(-1,1))}\lesssim (1+2^{u+i})^{-M+2}\|f\|_{L^q}.$$
		Since $M$ is arbitrary, this gives us that
		$$\|\widetilde{\Delta}^\ast_i f\|_{L^q(E_u(-1,1))}\lesssim \min(1,2^i)(1+2^{u+i})^{-M'}\|f\|_{L^q},$$ as desired. 
	\end{proof}

	\section{Proof of theorem}
	By inclusion of Schatten classes, it is sufficient to consider $p > 2$. By Theorem~\ref{thm:besov}, we need to show that $B_{p,p}^{1/p}(\mathbb{C}_+^2) \subset \BMO_{\produ}$. By the duality $(H^1(\mathbb{C}_+^2))^\ast \simeq P_{++}\BMO_{\textrm{prod}}$, the Hahn--Banach theorem then implies the existence of a bounded symbol. The very same duality, together with the atomic decomposition given by Theorem~\ref{thm:atomic}, in turn says that it is equivalent to show that 
	\[
	|\langle f, P_{++} a \rangle| = |\langle f, a \rangle|
	\lesssim
	\|f\|_{B^{1/p}_{p,p}}
	\]
	holds uniformly for every $L^q$-normalized atom $a$ and $f \in B^{1/p}_{p,p}(\mathbb{C}_+^2)$, where $q < 2$ is the dual index of $p$.  We therefore fix such an atom $a$, supported in an open set $\Omega$ of finite measure, with decomposition
	\[
	a=\sum_{R\in M_1(\Omega)} a_R + \sum_{R\in M_2(\Omega)} a_R
	\]
	satisfying \eqref{eq:qnormalization}; we fix a $\delta > 0$ so that 
	$$\beta:=\delta(p-1)<1.$$ 
	By appealing to a symmetric argument afterwards, we may of course assume that the first sum is empty, so that 
	$$a=  \sum_{R\in M_2(\Omega)} a_R.$$
	
	Recall our notation for product annuli from Section~\ref{subsec:anndecay}, $E_{u, v}(R) = E_u(I) \times E_v(J)$, and note that for every $R$, the sets $\{E_{u,v}(R)\}_{u,v \geq 0}$ tile $\R^2$ up to sets of measure zero. We begin with the following lemma.
	\begin{lm}
		We have that
		\begin{equation*}
			|\langle f, a_R \rangle|
			\lesssim
			|R|^{1/p}\,\|a_R\|_{L^q}
			\sum_{u,v\ge 0}2^{-(u+v)/q}
			\Bigg(\int_{E_{u,v}(R)} S_p(f)(x)^p\,dx\Bigg)^{1/p},
		\end{equation*}
		with an implied constant depending only on $q$.
	\end{lm}
	
	\begin{proof}
		We use the reproducing formula \eqref{eq:Caldron},
		\[
		\langle f, a_R \rangle
		=
		\sum_{i,j}\langle \Delta_{i,j}f, \widetilde\Delta_{i,j}^\ast a_R \rangle,
		\]
		split the integral according to the sets $E_{u,v}(R)$ and apply Hölder, which gives
		\[
		|\langle \Delta_{i,j}f, \widetilde\Delta_{i,j}^\ast a_R \rangle|
		\leq
		\sum_{u,v\ge 0}
		\|\Delta_{i,j}f\|_{L^p(E_{u,v}(R))}
		\|\widetilde\Delta_{i,j}^\ast a_R\|_{L^q(E_{u,v}(R))}.
		\]
		and, by Lemma~\ref{lem:annulardecay},
		\[
		|\langle \Delta_{i,j}f, \widetilde\Delta_{i,j}^\ast a_R \rangle|
		\lesssim
		\|a_R\|_{L^q}
		\sum_{u,v\ge 0}
		m_i(I)m_j(J)\,\theta_u(i,I)\,\theta_v(j,J)\,
		\|\Delta_{i,j}f\|_{L^p(E_{u,v}(R))}.
		\]
		By discrete H\"older, we thus have
		\[
		|\langle f, a_R \rangle|
		\lesssim
		\|a_R\|_{L^q}
		\sum_{u,v\ge 0}
		A_u(I)\,A_v(J)\,
		\Bigg(\sum_{i,j}2^i2^j\|\Delta_{i,j}f\|_{L^p(E_{u,v}(R))}^p\Bigg)^{1/p},
		\]
		where
		\[
		A_u(I) =
		\Bigg(\sum_i [2^{-i/p}m_i(I)\theta_u(i,I)]^q\Bigg)^{1/q}.
		\]
		To finish the proof, we thus need to show that
		$$A_u(I)\lesssim |I|^{1/p}2^{-u/q}.$$
		We will verify this explicitly for $u\geq 4$; the case $u=0$ is similar. First let us observe that
		$$A_u(I)=|I|^{1/p}A_u(-1/2,1/2),$$ and thus it suffices to show that 
		$$A_u(-1/2,1/2)^q =
		\sum_i 2^{-iq/p}\min(1,2^i)^q(1+2^{i+u})^{-Mq}\lesssim 2^{-u}.$$ 
		First, for $i\leq -u$ we have that
		$$\sum_{i \leq -u}2^{-iq/p}\min(1,2^i)^q(1+2^{i+u})^{-Mq} \leq \sum_{i\leq -u} 2^{-iq/p}2^{iq}=\sum_{i\leq -u} 2^{i}\approx 2^{-u}.$$
		Now for the case $-u<i\leq 0$ we can use the fact that $(1+2^{i+u})^{-Mq}\lesssim 2^{-Mq(i+u)}$ to get
		\begin{eqnarray} 
			\sum_{-u<i\leq 0}2^{-iq/p}\min(1,2^i)^q(1+2^{i+u})^{-Mq}&\lesssim& \sum_{-u<i\leq 0}2^{-iq/p}2^{iq}2^{-Mq(i+u)} \nonumber \\
			&=&2^{-Mqu}\sum_{-u<i\leq 0}2^{i(1-qM)}\approx 2^{-u}. \nonumber 
		\end{eqnarray} 
		Finally, for $i>0$, using the same estimate,
		\begin{eqnarray}
			\sum_{i>0}2^{-iq/p}\min(1,2^i)^q(1+2^{i+u})^{-Mq}&\lesssim& \sum_{i>0}2^{-iq/p}2^{-Mq(i+u)}   \nonumber \\
			&\lesssim& 2^{-Mqu}\sum_{i>0}2^{-iq(1/p+M)}\approx 2^{-Mqu}. \nonumber 
		\end{eqnarray}
		Combining all the estimates we get that 
		$$A_u(-1/2,1/2)^q\lesssim 2^{-u}+2^{-Mqu}+2^{-u}\approx 2^{-u},$$
		as desired.			
	\end{proof}
	
	Summing over $R \in M_2(\Omega)$ and applying H\"older, we therefore find that
	\begin{multline*}
		|\langle f, a \rangle| \lesssim \sum_{u, v \ge 0} 2^{-(u+v)/q} \Bigg(\sum_{R\in M_2(\Omega)} \gamma_1(R)^{-\delta} \|a_R\|_q^q\Bigg)^{1/q}
		\Bigg(\sum_{R\in M_2(\Omega)} \gamma_1(R)^{\delta \frac{p}{q}}\,|R|
		\int_{E_{u,v}(R)} S_p(f)^p\Bigg)^{1/p} \\ \lesssim
		|\Omega|^{1/q - 1}\sum_{u, v \ge 0} 2^{-(u+v)/q} 
		\Bigg(\sum_{R\in M_2(\Omega)} \gamma_1(R)^{\beta}\,|R|
		\int_{E_{u,v}(R)} S_p(f)^p\Bigg)^{1/p} \\
		= |\Omega|^{1/q - 1}\sum_{u, v \ge 0} 2^{-(u+v)/q} 
		\Bigg(\int_{\R^2} S_p(f)^p(x)  \sum_{\substack{R\in M_2(\Omega) \\ x \in E_{u,v}(R)}} \gamma_1(R)^{\beta}\,|R| \, dx
		\Bigg)^{1/p},
	\end{multline*}
	where we used the $L^q$-normalization \eqref{eq:qnormalization} in the second step, and Fubini in the third.
	
	Thus the proof is finished if we can establish that
	\begin{equation} \label{eq:maingeometric}
		\sum_{\substack{R = I \times J \in M_2(\Omega)\\ x\in 2^u I \times 2^v J}}
		\gamma_1(R)^\beta |R|
		\leq C_\beta
		2^{u+v}|\Omega|,
	\end{equation}
	with a constant that only depends on $\beta < 1$, since this allows us to conclude that
	$$|\langle f, a \rangle| \lesssim \sum_{u,v \geq 0} 2^{(1/p - 1/q)(u+v)} \|S_p(f)\|_{L^p} \lesssim \|f\|_{B_{p,p}^{1/p}},$$
	with a final implied constant depending only on $p$. To carry this out we require the following elementary dyadic counting lemma.
	\begin{lm} \label{lem:dyadcount}
		Let $g \in L^1(\R)$ be non-negative, $g \geq 0$. Then, for $\lambda \geq 1$ and $x \in \R$,
		\begin{equation*}
			\sum_{\substack{J \in \mathcal{D}\\ x\in \lambda J}}
			|J|\,\inf_{J} g 
			\leq
			6\lambda \|g\|_{L^1(\R)}.
		\end{equation*}
	\end{lm}
	\begin{proof}
		Fix $x$ and $\lambda$. For $t > 0$, let $E_t:=\{y \, : \, g(y) >t\}$, and let $\mathcal M_t$ be the collection of maximal dyadic intervals contained in $E_t$. For $L\in \mathcal M_t$, there are at most $3\lambda$ intervals $J \in \mathcal{D}$ such that $x \in \lambda J$ and $|J| = 2^{-m}|L|$, $m = 0, 1, \ldots$.  Thus
		\[
		\sum_{\substack{J\subset L\\ x\in \lambda J}} |J|
		\leq 6\lambda |L|,
		\]
		and summing over \(\mathcal M_t\), noting that sets in $L\in \mathcal M_t$ are pairwise disjoint,
		\[
		\sum_{\substack{J\subset E_t \\ x\in \lambda J}} |J|
		\leq 6\lambda \sum_{L\in \mathcal M_t}|L|
		\le 6 \lambda |E_t|.
		\]
		To finish, observe now that
		\[
		\sum_{x\in \lambda J}|J|\inf_J g = 	\int_0^\infty
		\sum_{\substack{ \inf_J g>t \\ x\in \lambda J}} |J|\,dt \leq \int_0^\infty
		\sum_{\substack{ J \subset E_t \\ x\in \lambda J}} |J|\,dt \leq 6 \lambda \int_0^\infty |E_t| \, dt = 6\lambda \|g\|_{L^1(\R)},
		\]
		where we used Fubini in the first (and in the final) step.
	\end{proof}
	
	Returning to the proof of \eqref{eq:maingeometric}, fix an interval $J \in \mathcal{D}$ for which there exist $I \in \mathcal{D}$ such that $I \times J \in M_2(\Omega)$, and fix a corresponding maximal $\hat I \in \mathcal{D}$ for which $\hat I \times J \subset \wtOmega$. Among those rectangles with embeddedness constant
	$$2^k \leq \gamma_1(R) < 2^{k+1},$$ there are at most $6  \cdot 2^u$ rectangles $R = I \times J \subset \hat I \times J $ for which $x_1 \in 2^u I$. Therefore 
	\[ \sum_{\substack{R = I \times J \subset \hat I \times J \\ x_1 \in 2^u I}} \gamma_1(R)^\beta |R|
	\lesssim
	2^u \sum_{k =0}^\infty 2^{\beta k} 2^{-k}|\hat I||J| \lesssim
	2^u |\hat I||J|.
	\]
	Summing over all $\hat{I}$ and $J$ therefore yields
	\begin{equation*}
		\sum_{\substack{R = I \times J \in M_2(\Omega)\\ x\in 2^u I \times 2^v J}}
		\gamma_1(R)^\beta |R|
		\lesssim 2^u \sum_{\substack{J \in \mathcal{D} \\ x_2 \in 2^v J}} |J| |F_J|,
	\end{equation*}
	where
	\[
	F_J = \bigcup\{\hat I:\ \hat I\times J\subset \wtOmega,\ \hat I \text{ maximal dyadic}\}.
	\]
	By definition, $F_J$ is contained in the slice $\{x_1 \, :\, (x_1,y) \in \wtOmega\}$ for every $y \in J$. Therefore 
	$$|F_J| \leq \inf_J g,$$
	where $g(y) = |\{x_1 \, :\, (x_1,y) \in \wtOmega\}|$. We can therefore apply Lemma~\ref{lem:dyadcount}, yielding that 
	\[2^u \sum_{\substack{J \in \mathcal{D} \\ x_2 \in 2^v J}} |J| |F_J| \lesssim 2^{u+v} \| g\|_{L^1} = 2^{u+v} |\wtOmega| \lesssim 2^{u+v} |\Omega|.\]
	This establishes \eqref{eq:maingeometric} and finishes the proof of the theorem.

	\bibliographystyle{plain}
	\bibliography{BibSchatten1}

\begin{thebibliography}{10}

\bibitem{AMPS23}
Nicola Arcozzi, Pavel Mozolyako, Karl-Mikael Perfekt, and Giulia Sarfatti.
\newblock Bi-parameter potential theory and {C}arleson measures for the
  {D}irichlet space on the bidisc.
\newblock {\em Discrete Anal.}, pages Paper No. 22, 57, 2023.

\bibitem{bampouras2026nehari}
Konstantinos Bampouras.
\newblock Nehari's theorem and {H}ardy's inequality for {P}aley-{W}iener
  spaces.
\newblock {\em J. Funct. Anal.}, 290(10):Paper No. 111408, 26, 2026.

\bibitem{BP25}
Konstantinos Bampouras and Karl-Mikael Perfekt.
\newblock Besov spaces and {S}chatten class {H}ankel operators for {H}ardy and
  {P}aley--{W}iener spaces in higher dimensions.
\newblock {\em arXiv:2409.04184}, 2025.

\bibitem{MR4227573}
Marcus Carlsson and Karl-Mikael Perfekt.
\newblock Nehari's theorem for convex domain {H}ankel and {T}oeplitz operators
  in several variables.
\newblock {\em Int. Math. Res. Not. IMRN}, (5):3331--3361, 2021.

\bibitem{ChaFef82}
Sun-Yung~A. Chang and Robert Fefferman.
\newblock The {C}alder\'on-{Z}ygmund decomposition on product domains.
\newblock {\em Amer. J. Math.}, 104(3):455--468, 1982.

\bibitem{Fef86}
Robert Fefferman.
\newblock Some recent developments in {F}ourier analysis and {$H^p$} theory on
  product domains. {II}.
\newblock In {\em Function spaces and applications ({L}und, 1986)}, volume 1302
  of {\em Lecture Notes in Math.}, pages 44--51. Springer, Berlin, 1988.

\bibitem{FefStein82}
Robert Fefferman and Elias~M. Stein.
\newblock Singular integrals on product spaces.
\newblock {\em Adv. in Math.}, 45(2):117--143, 1982.

\bibitem{HanLuZhao10}
Y.~Han, G.~Lu, and K.~Zhao.
\newblock Discrete {C}alder\'on's identity, atomic decomposition and
  boundedness criterion of operators on multiparameter {H}ardy spaces.
\newblock {\em J. Geom. Anal.}, 20(3):670--689, 2010.

\bibitem{HLPW21}
Yongsheng Han, Ji~Li, M.~Cristina Pereyra, and Lesley~A. Ward.
\newblock Atomic decomposition of product {H}ardy spaces via wavelet bases on
  spaces of homogeneous type.
\newblock {\em New York J. Math.}, 27:1173--1239, 2021.

\bibitem{MR2263964}
Henry Helson.
\newblock Hankel forms and sums of random variables.
\newblock {\em Studia Math.}, 176(1):85--92, 2006.

\bibitem{HTV21}
Irina Holmes, Sergei Treil, and Alexander Volberg.
\newblock Dyadic bi-parameter repeated commutator and dyadic product {BMO}.
\newblock {\em arXiv:2101.00763}, 2021.

\bibitem{Journe85}
Jean-Lin Journ\'e.
\newblock Calder\'on-{Z}ygmund operators on product spaces.
\newblock {\em Rev. Mat. Iberoamericana}, 1(3):55--91, 1985.

\bibitem{MPV22}
Pavel Mozolyako, Georgios Psaromiligkos, Alexander Volberg, and Pavel
  Zorin-Kranich.
\newblock Carleson embedding on the tri-tree and on the tri-disc.
\newblock {\em Rev. Mat. Iberoam.}, 38(7):2069--2116, 2022.

\bibitem{Pipher86}
Jill Pipher.
\newblock Journ\'e's covering lemma and its extension to higher dimensions.
\newblock {\em Duke Math. J.}, 53(3):683--690, 1986.

\bibitem{MR2097606}
Sandra Pott and Martin~P. Smith.
\newblock Paraproducts and {H}ankel operators of {S}chatten class via
  {$p$}-{J}ohn-{N}irenberg theorem.
\newblock {\em J. Funct. Anal.}, 217(1):38--78, 2004.

\bibitem{Rud69}
Walter Rudin.
\newblock {\em Function theory in polydiscs}.
\newblock W. A. Benjamin, Inc., New York-Amsterdam, 1969.

\end{thebibliography}
	
\end{document}